\documentclass[a4paper,11pt]{article}
\usepackage{amssymb,amsmath,amsthm}
\usepackage{graphics,graphicx,color}
\usepackage{srcltx, bm}
\DeclareGraphicsExtensions{.eps} \textwidth=169mm
\definecolor{darkred}{rgb}{0.9,0.1,0.1}

\newtheorem{proposition}{Proposition}[section]
\newtheorem{theorem}{Theorem}[section]
\newtheorem{lemma}{Lemma}[section]
\newtheorem{corollary}{Corollary}
\newtheorem{remark}{Remark}[section]
{\rm}
\definecolor{darkred}{rgb}{0.9,0.1,0.1}

\def\eps{\varepsilon}
\def\eps{\varepsilon}
\definecolor{darkred}{rgb}{0.9,0.1,0.1}

\begin{document}

\title{Homogenization of biased convolution type operators}

\author{A. Piatnitski\thanks{Institute for Information Transmission Problems of RAS,
 19, Bolshoi Karetnyi per. build.1,
 127051 Moscow, Russia  and  The Arctic University of Norway, Campus in Narvik,
 P.O.Box 385, Narvik 8505, Norway (apiatnitski@gmail.com) }
\and
\setcounter{footnote}{2}
E. Zhizhina\thanks{Institute for Information Transmission Problems of RAS,
 19, Bolshoi Karetnyi per. build.1,
 127051 Moscow, Russia  (ejj@iitp.ru)} }

\date{}
\maketitle

\begin{abstract}
This paper deals with homogenization of parabolic problems  for  integral convolution type operators with a
non-symmetric jump kernel in a periodic elliptic medium.   It is shown that the homogenization result holds in moving
coordinates. We determine the corresponding effective velocity and prove that the limit operator is a second order
parabolic operator with constant coefficients.  We also consider the behaviour of the effective velocity in the case of small antisymmetric perturbations of a symmetric kernel. 
\end{abstract}

\section{Introduction}

The paper deals with homogenization of parabolic problems for an integral convolution type operator of the form
\begin{equation}\label{L_u}
(L u)(x) \ = \ \int\limits_{\mathbb R^d} a(x-y) \mu(x,y) (u(y) - u(x)) dy
\end{equation}
with a non-symmetric jump kernel $a(z)$ and a periodic positive function $\mu(x,y)$.\\
In our previous work  \cite{PZh}  we considered an integral convolution type operator defined by
\begin{equation}\label{L_u_prev}
(L u)(x) \ = \ \lambda(x)\int\limits_{\mathbb R^d} a(x-y) \mu(y) (u(y) - u(x)) dy
\end{equation}
under the assumption that $\lambda(x)$ and $\mu(y)$ are bounded positive periodic functions characterizing the properties of the medium, and $a(z)$ is the jump kernel being a positive integrable function such that $a(-z)=a(z)$.
We then made a diffusive scaling of this operator
\begin{equation}\label{L_u_biseps}
(L^\eps u)(x) \ = \ \eps^{-d-2}\lambda\Big(\frac{x}{\eps}\Big) \int\limits_{\mathbb R^d} a\Big(\frac{x-y}{\eps}\Big) \mu\Big(\frac{y}{\eps}\Big) (u(y) - u(x)) dy,
\end{equation}
where $\eps$ is a positive scaling factor, $\eps\ll 1$.  Then we proved
the homogenization result for the operators $L^\eps$. More precisely, we proved
that the family $L^\eps$ converges, as $\eps\to0$, to a second order
divergence form elliptic operator with constant coefficient in the so-called $G$-topology that is
for any $m>0$ the family of operators $(-L^\eps+m)^{-1}$ converges strongly in $L^2(\mathbb R^d)$
to the operator $(-L^0+m)^{-1}$ where $L^0=\Theta^{ij}\frac{\partial^2}{\partial x^i\partial x^j}$
with a positive definite constant matrix $\Theta$.

In this work we consider homogenization problems for convolution type operators $L$ with a kernel of the form
$a(x-y) \mu(x,y)$, where the function $a(z)$ is not assumed to be even. More precisely, we assume that $a(z)$
is the generic non-negative integrable function in $\mathbb R^d$ that has finite second moments. Concerning the
coefficient  $\mu(x,y)$ we assume that this function is periodic both in $x$ and $y$ and satisfies the estimates
$0<\alpha_1\leq\mu(x,y)\leq\alpha_2$ for some positive constants $\alpha_1$ and $\alpha_2$.

In this framework it is natural to study the evolution version of the corresponding homogenization problem. 
Namely,
we are going to investigate the limit behaviour of a solution to the following parabolic equation:
\begin{equation}\label{pbm_evol_eps}
\partial_t u(x,t)-(L^\eps u)(x,t) \ = \ 0, \qquad u(x,0)=u_0(x).
\end{equation}
Clearly, under the above conditions on $a$ and $\mu$ the effective velocity need not be zero. This raises the
following two natural problems:  to determine the effective velocity, and to obtain homogenization results in the
corresponding moving coordinates.
In the paper we address both this questions. The main homogenization results are formulated in Theorem \ref{Theorem1} below.

We also consider a small antisymmetric perturbation of a symmetric kernel and study how the effective velocity and other effective characteristics
react on this small perturbation. These results are summarized in Lemma \ref{ER}.  In particular, we prove that the Einstein relation holds for the perturbation of special structure.

It is interesting to compare the effective behaviour of parabolic equations for nonlocal non-symmetric
convolution type operators and for differential operators of convection-diffusion type.
Homogenization problems for non-stationary convection-diffusion equations in periodic media
have been investigated in the works  \cite{AlOr},  \cite{DoPi}.
It was shown in  \cite{AlOr},  \cite{DoPi}, that the homogenization takes place in the moving coordinates
$X(t)=x-\frac{b}\eps \, t$ with an appropriate constant vector $ b$.
For an elliptic  diffusion in a periodic environment and in a random ergodic environment with a finite range of
dependence the Einstein relation was proved in \cite{GaMaPi},   for a random walk with i.i.d. conductances
it was justified in \cite{GaGuNa}.

\section{Problem setup and main results}

In this section we provide all the conditions on the coefficients of operator $L$
and then formulate our main results.

Regarding the function $a(\cdot)$ we assume that
\begin{equation}\label{M1}
a(z) \in  L^{1}(\mathbb R^d), \; \; a(z) \ge 0, \; \; \hat a(\eta) \in  L^{2}(\mathbb T^d),
\end{equation}
and
\begin{equation}\label{M2}
\| a \|_{L^1(\mathbb R^d)}  = \int\limits_{\mathbb R^d} a(z) \, dz = a_1 >0; \qquad \int\limits_{\mathbb R^d} |z|^2 a(z) \, dz < \infty.
\end{equation}

The function $\mu(x,y)$ is periodic in both variables and bounded from above and from below:
\begin{equation}\label{lm}
0< \alpha_1 \le  \mu(x,y) \le \alpha_2 < \infty.
\end{equation}
From now on we identify periodic functions in $\mathbb R^d$ with functions defined on the torus $\mathbb{T}^d = {\mathbb R^d}/{\mathbb Z^d} $. The operator $L$ is a bounded not necessary symmetric operator in $L^2(\mathbb R^d)$, see \cite{PZh}.

In what follows we also use the function
$$
\hat a(\eta) = \sum\limits_{k \in Z^d} a(\eta +k), \; \eta \in  {\mathbb T^d}.
$$
Notice that $\hat a$ is non-negative, and $\|\hat a\|_{L^1(\mathbb T^d)}=\|a\|_{L^1(\mathbb R^d)}$.

Let us consider the following evolution operator
$$
H = \frac{\partial}{\partial t} - L,
$$
with 
$L$  defined in \eqref{L_u}.
Then, performing the change of variables $x \to \varepsilon x, \; t \to \varepsilon^2 t$, we obtain the family of rescaled operators
\begin{equation}\label{A_eps}
H^\varepsilon u = \frac{\partial u}{\partial t} -  L^{\varepsilon} u, \quad \mbox{where} \;\;
(L^\varepsilon u)(x,t) \ = \ \frac{1}{\varepsilon^{d+2}} \int\limits_{\mathbb R^d} a\Big(\frac{x-y}{\eps}\Big) \mu\Big(\frac{x}{\eps}, \frac{y}{\eps}\Big) (u(y,t) - u(x,t)) dy.
\end{equation}

The main result of this paper is the following homogenization theorem.

\begin{theorem}\label{Theorem1}
Assume that the functions $a(z)$ and $\mu(x,y)$ satisfy conditions \eqref{M1} - \eqref{lm}.

Let $u^{\varepsilon}(x,t)$ be the solution of the evolution problem
\begin{equation}\label{PE_eps}
\frac{\partial u^{\varepsilon}}{\partial t} =  L^{\varepsilon} u^{\varepsilon}, \quad u^{\varepsilon}(x,0) = \varphi(x), \quad \varphi \in L^2(\mathbb R^d),
\end{equation}
and $u^0(x,t)$ be the solution of a parabolic problem
\begin{equation}\label{PE}
\frac{\partial u^{0}}{\partial t} = \Theta \cdot \nabla \nabla u^{0}, \quad u^{0}(x,0) = \varphi(x), \quad \varphi \in L^2(\mathbb R^d).
\end{equation}
Then
there exist a vector $b \in {\mathbb R^d}$ and a positive definite constant matrix
$\Theta$ such that for any $T>0$:
\begin{equation}\label{T1}
\| u^{\varepsilon}\big( x+\frac{b}{\eps}\ t, \ t \big) - u^0 (x, t) \|_{L^{\infty}((0,T),\ L^2(\mathbb R^d) )} \to 0 \quad \mbox{ as } \; \varepsilon \to 0.
\end{equation}

\end{theorem}

\medskip

Observe that
\begin{equation}\label{T1bis}
\| u^{\varepsilon}\big( x+\frac{b}{\eps}\ t, \ t \big) - u^0 (x, t) \|_{L^{\infty}((0,T),\ L^2(\mathbb R^d) )} = \| u^{\varepsilon}\big( x, \ t \big) - u^0 (x-\frac{b}{\eps}\ t, t) \|_{L^{\infty}((0,T),\ L^2(\mathbb R^d) )}
\end{equation}

\section{Correctors and auxiliary cell problems}
\label{s_corr}

In this section we approximate a solution $u^{\varepsilon}$ of problem (\ref{PE_eps}) using an ansatz constructed in terms of a solution $u^0$ of the limit problem (\ref{PE}) with the same initial condition $\varphi$. To this end we consider auxiliary periodic problems, whose solutions (the so-called correctors) are used in the construction of this ansatz and define the coefficients $\Theta$  of effective operator in (\ref{PE}). We first deal with functions from the Schwartz space ${\cal{S}}(\mathbb R^d)$ that are smooth in $t$ on any interval $t \in (0,T)$.

For a given $ u \in C^{\infty}((0,T),{\cal{S}}(\mathbb R^d))$ we introduce the following ansatz:
\begin{equation}\label{w_eps}
w^{\varepsilon}(x,t) \ = \ u(x - \frac{b}{\varepsilon} \, t,t) + \varepsilon \varkappa_1 (\frac{x}{\varepsilon})\cdot \nabla u(x - \frac{b}{\varepsilon}\, t,t) + \varepsilon^2 \varkappa_2 (\frac{x}{\varepsilon})\cdot \nabla \nabla u(x - \frac{b}{\varepsilon}\, t,t),
\end{equation}
where the vector $b \in \mathbb R^d$ and correctors $\varkappa_1 \in (L^2(\mathbb T^d))^d$  and $\varkappa_2 \in (L^2(\mathbb T^d))^{d^2}$ (a vector function $\varkappa_1$ and a matrix function $\varkappa_2$) will be defined below.

\begin{lemma}\label{ml}   Assume that $u \in C^{\infty}((0,T),{\cal{S}}(\mathbb R^d))$. Then there exist functions $\varkappa_1 \in (L^2(\mathbb T^d))^d$  and $\varkappa_2 \in (L^2(\mathbb T^d))^{d^2}$, a vector $b \in \mathbb R^d$ and a positive definite matrix $\Theta$ such that  for the function $w^{\varepsilon}$ defined by \eqref{w_eps} we obtain
\begin{equation}\label{mle}
H^{\varepsilon} w^{\varepsilon}(x,t) \ :=  \frac{\partial w^\varepsilon}{\partial t} -  L^{\varepsilon} w^\varepsilon \ = \  \Big(\frac{\partial u}{\partial t}(x^\varepsilon,t) -
 \Theta \cdot \nabla \nabla u(x^\varepsilon,t) \ + \ \phi^\varepsilon(x^\varepsilon,t)\Big)|_{_{\, x^\varepsilon = x - \frac{b}{\varepsilon} \, t}},
\end{equation}
where
\begin{equation}\label{mle-rest}
\lim\limits_{\varepsilon \to 0}\| \phi^\varepsilon \|_{L^{\infty}((0,T),\ L^2(\mathbb R^d) )} = 0.
\end{equation}
\end{lemma}

\begin{proof}
Substituting the expression on the right-hand side of \eqref{w_eps} for $u$ in \eqref{A_eps} and using the notation $x^\varepsilon = x - \frac{b}{\varepsilon} \, t$ we get
\begin{equation}\label{Aw}
\begin{array}{rl}
\displaystyle
H^{\varepsilon}w^{\varepsilon}(x,t) \
= \ \frac{\partial w^{\varepsilon}(x, t)}{\partial t} -  \frac{1}{\varepsilon^{d+2}} \int\limits_{\mathbb R^d} a\big(\frac{x-y}{\eps}\big) \mu\big(\frac{x}{\eps}, \frac{y}{\eps}\big) (w^{\varepsilon}(y,t) - w^{\varepsilon}(x,t)) dy
\\[3mm] \displaystyle = \
\big( - \frac{b}{\varepsilon} \big) \cdot \nabla u (x^\varepsilon, t)  + \frac{\partial u}{\partial t} (x^\varepsilon, t) + \varepsilon \varkappa_1 \big(\frac{x}{\varepsilon}\big) \otimes\big( - \frac{b}{\varepsilon} \big) \cdot \nabla \nabla u (x^\varepsilon, t)
\\[3mm] \displaystyle + \
 \varepsilon \varkappa_1 \big(\frac{x}{\varepsilon}\big)  \cdot \nabla \frac{\partial u}{\partial t} (x^\varepsilon, t) +
 \varepsilon^2 \varkappa_2 \big(\frac{x}{\varepsilon}\big) \otimes\big( - \frac{b}{\varepsilon} \big) \cdot \nabla \nabla \nabla u (x^\varepsilon, t) + \varepsilon^2 \varkappa_2 \big(\frac{x}{\varepsilon}\big)  \cdot \nabla \nabla \frac{\partial u}{\partial t} (x^\varepsilon, t)
\\[3mm] \displaystyle - \
\frac{1}{\varepsilon^{d+2}} \int\limits_{\mathbb R^d} a \big( \frac{x-y}{\varepsilon} \big) \mu \big( \frac{x}{\varepsilon}, \frac{y}{\varepsilon} \big)
\bigg\{ u(y^\varepsilon,t)+ \varepsilon \varkappa_1 \big(\frac{y}{\varepsilon}\big)\cdot \nabla u(y^\varepsilon,t) +
\\[3mm] \displaystyle +\
\varepsilon^2 \varkappa_2 \big(\frac{y}{\varepsilon}\big)\cdot \nabla \nabla u(y^\varepsilon,t) -
u(x^\varepsilon,t)-\varepsilon \varkappa_1 \big(\frac{x}{\varepsilon} \big)\cdot \nabla u(x^\varepsilon,t) - \varepsilon^2 \varkappa_2 \big(\frac{x}{\varepsilon}\big)\cdot \nabla \nabla u(x^\varepsilon,t) \bigg\} dy,\\
\end{array}
\end{equation}
where the symbol $\otimes$ stands for tensor product, in particular 
$$ 
\varkappa_2 \big(\frac{x}{\varepsilon}\big) \otimes\big( - \frac{b}{\varepsilon} \big) \cdot \nabla \nabla \nabla u= \varkappa^{ij}_2 \big(\frac{x}{\varepsilon}\big) \big( - \frac{b^k}{\varepsilon} \big)\, \partial_{x^i} \partial_{x^j} \partial_{x^k} u.
$$ 
Here and in the sequel we assume
summation over repeated indices.

We collect the terms in \eqref{Aw} that give the main contribution on the right hand side of equality \eqref{mle};
the higher order  terms form the  remainder $\phi^\varepsilon$. We do this separately for $\frac{\partial w^{\varepsilon}}{\partial t}$
and for $L^\varepsilon w^\varepsilon$. For $\frac{\partial w^{\varepsilon}}{\partial t}$ we obtain 
\begin{equation}\label{Aw-t}
\frac{\partial w^{\varepsilon}(x, t)}{\partial t} =
\big( - \frac{b}{\varepsilon} \big) \cdot \nabla u (x^\varepsilon, t)  + \frac{\partial u}{\partial t} (x^\varepsilon, t) + \varepsilon \varkappa_1 \big(\frac{x}{\varepsilon}\big)\otimes \big( - \frac{b}{\varepsilon} \big) \cdot \nabla \nabla u (x^\varepsilon, t) + \phi_\varepsilon^{(0)}(x,t),
\end{equation}
with
\begin{equation}\label{Aw-t-r}
\phi_\varepsilon^{(0)}(x,t) =  \varepsilon \varkappa_1 \big(\frac{x}{\varepsilon}\big)  \cdot \nabla \frac{\partial u}{\partial t} (x^\varepsilon, t) +  \varepsilon^2 \varkappa_2 \big(\frac{x}{\varepsilon}\big)\otimes \big( - \frac{b}{\varepsilon} \big) \cdot \nabla \nabla \nabla u (x^\varepsilon, t) + \varepsilon^2 \varkappa_2 \big(\frac{x}{\varepsilon}\big)  \cdot \nabla \nabla \frac{\partial u}{\partial t} (x^\varepsilon, t).
\end{equation}
After change of variables $z = \frac{x-y}{\varepsilon} = \frac{x^\varepsilon - y^\varepsilon}{\varepsilon}$ we get
\begin{equation}\label{ml_1}
\begin{array}{l}\displaystyle
\!\!\!(L^{\varepsilon} w^{\varepsilon})(x,t) \ = \ \frac{1}{\varepsilon^{2}} \int\limits_{\mathbb R} dz \  a (z)  \mu \big(\frac{x}{\varepsilon}, \frac{x}{\varepsilon} -z \big) \bigg\{ u(x^\varepsilon -\varepsilon z,t)
+ \varepsilon \varkappa_1 \big(\frac{x}{\varepsilon}-z \big)\cdot \nabla u (x^\varepsilon -\varepsilon z,t)
\\[4mm] \displaystyle
+\,\varepsilon^2 \varkappa_2 \big( \frac{x}{\varepsilon}-z \big)\cdot \nabla \nabla u(x^\varepsilon -\varepsilon z,t)
 - u(x^\varepsilon,t) -\varepsilon \varkappa_1 \big( \frac{x}{\varepsilon} \big)\cdot\nabla u(x^\varepsilon,t) - \varepsilon^2 \varkappa_2 \big(\frac{x}{\varepsilon} \big)\cdot \nabla \nabla u(x^\varepsilon,t) \bigg\}.
\end{array}
\end{equation}
Using the following relations
\begin{eqnarray*}
&&u(y) \ = \ u(x) + \int_0^1 \frac{\partial}{\partial q} \ u(x+(y-x)q) \ dq \ = \ u(x) + \int_0^1 \nabla u (x+  (y-x)q) \cdot (y-x) \ dq,
\\
&&u(y) \ = \ u(x) + \nabla u(x) \cdot (y-x) + \int_0^1  \nabla \nabla u(x+(y-x)q) (y-x)\cdot (y-x) (1-q) \ dq
\end{eqnarray*}
based on the integral form of a remainder in the Taylor expansion and  being valid for any $x, y \in \mathbb R^d$, we  rearrange (\ref{ml_1}) as follows
\begin{equation*}
\begin{array}{l}
(L^{\varepsilon} w^{\varepsilon})(x,t) =
\\[2mm]
\displaystyle
  \frac{1}{\varepsilon^{2}}\!\! \int\limits_{\mathbb R^d}\! dz \, a(z)   \mu \big( \frac{x}{\varepsilon}, \frac{x}{\varepsilon}\! -\!z \big)\! \bigg\{( u(x^\varepsilon,t) - \varepsilon z\cdot\nabla u(x^\varepsilon,t) + \varepsilon^2\! \int\limits_0^{1}\!  \nabla \nabla u(x^\varepsilon-\varepsilon zq, t)\cdot z\!\otimes\!z\, (1\!-\!q) \, dq
\\[5mm]
+\, \displaystyle\varepsilon \varkappa_1 \big(\frac{x}{\varepsilon}-z \big)\cdot \Big( \nabla u(x^\varepsilon,t)\!-\!\varepsilon  \nabla \nabla u (x^\varepsilon,t)\,z + \varepsilon^2 \int_0^{1}  \nabla \nabla \nabla u(x^\varepsilon-\varepsilon zq, t) z\!\otimes\!z (1-q) \ dq \Big) \\[4mm]
+\,\displaystyle
\varepsilon^2 \varkappa_2 \big( \frac{x}{\varepsilon}-z \big)\cdot \nabla \nabla u(x^\varepsilon-\varepsilon z,t) \ - \
u(x^\varepsilon,t)-\varepsilon \varkappa_1 \big( \frac{x}{\varepsilon} \big)\cdot \nabla u(x^\varepsilon,t) - \varepsilon^2 \varkappa_2 \big(\frac{x}{\varepsilon} \big)\cdot \nabla \nabla u(x^\varepsilon,t) \bigg\},
\end{array}
\end{equation*}
where
$$
\big\{  \nabla \nabla u(\cdot)z\big\}^i=\frac{\partial^2 u}{\partial x^i\partial x^j}(\cdot) z^j
\quad\hbox{and}\quad
\big\{\nabla \nabla \nabla u(\cdot) z\!\otimes\!z\big\}^i =\frac{\partial^3 u}{\partial x^i\partial x^j\partial x^k}(\cdot) z^jz^k.
$$
 Collecting power-like terms in the last relation
we obtain
\begin{equation}\label{K2_1}
\begin{array}{ll}
 (L^{\varepsilon}w^{\varepsilon})(x,t) \\[1.6mm]
\displaystyle
=\, \frac{1}{\varepsilon} \nabla u(x^\varepsilon,t)\! \cdot\! \int\limits_{\mathbb R^d}  \Big\{ -z + \varkappa_1 \big(\frac{x}{\varepsilon}-z \big) - \varkappa_1 \big(\frac{x}{\varepsilon}\big) \Big\}  a (z) \mu \big( \frac{x}{\varepsilon}, \frac{x}{\varepsilon} -z \big) \, dz
\\[1mm]
\displaystyle
 +\, \nabla \nabla u (x^\varepsilon,t)\!\cdot\!\!  \int\limits_{\mathbb R^d}\! \Big\{ \frac12 z\!\otimes\!z\! - z \!\otimes\!\varkappa_1 \big(\frac{x}{\varepsilon}\!-\!z \big) + \varkappa_2 \big( \frac{x}{\varepsilon}\! -\! z \big) \!- \varkappa_2 \big(\frac{x}{\varepsilon}\big)  \Big\} \  a (z) \mu \big(\frac{x}{\varepsilon}, \frac{x}{\varepsilon}\! -\!z \big) \, dz
\\[0.5mm]
 +\, \ \phi^{(L)}_\varepsilon (x,t)\\
\end{array}
\end{equation}
with
\begin{equation}\label{14}
\begin{array}{rl}
\phi^{(L)}_\varepsilon (x,t) \!\!&\displaystyle =\: \frac1{\varepsilon^{2}}\! \int\limits_{\mathbb R^d}\! dz \, a (z) \mu \big( \frac{x}{\varepsilon}, \frac{x}{\varepsilon}\! -\!z \big)
\bigg\{ \varepsilon^2\! \int\limits_0^{1}  \nabla \nabla u(x^\varepsilon-\varepsilon zq, t) \!\cdot\! z\!\otimes\!z \,(1-q) \ dq
\\[4mm]  & \displaystyle
 - \frac{\varepsilon^2}2 \nabla \nabla u(x^\varepsilon,t)\!\cdot\!  z\!\otimes\!z \
+\, \varepsilon^3 \varkappa_1 \big(\frac{x}{\varepsilon}\!-\!z \big)\!\cdot\! \int\limits_0^{1}\!  \nabla \nabla \nabla u(x^\varepsilon\!-\!\varepsilon zq, t) z\!\otimes\!z (1\!-\!q) \, dq  \,
\\[4mm]  & \displaystyle
- \, \varepsilon^3 \varkappa_2 \big(\frac{x}{\varepsilon}\!-\!z \big) \!\cdot\! \int\limits_0^{1}\!  \nabla \nabla \nabla u(x^\varepsilon\!-\!\varepsilon zq, t)z  \, dq\!  \bigg\}.
\end{array}
\end{equation}
Thus the remainder term $\phi^\varepsilon$ is the sum
\begin{equation}\label{fiplus}
\phi^\varepsilon \ = \ \phi^{(0)}_\varepsilon \ + \ \phi^{(L)}_\varepsilon.
\end{equation}


\begin{proposition}\label{fi_0} Let $u \in C^{\infty}\big( (0,T), {\cal{S}}(\mathbb R^d) \big)$
Then for the functions $ \phi^{(0)}_\varepsilon$ and $ \phi^{(L)}_\varepsilon$ given by \eqref{Aw-t-r} and \eqref{14} we have
\begin{equation}\label{fi}
\| \phi^{(L)}_\varepsilon \|\big._{\infty} \ \to \ 0 \quad \mbox{ and } \quad \| \phi^{(0)}_\varepsilon \|\big._{\infty} \ \to \ 0 \quad  \mbox{ as } \; \varepsilon \to 0,
\end{equation}
where $\| \cdot \|\big._{\infty}$ is the norm in $ L^{\infty}\big( (0,T), L^2 (\mathbb R^d) \big)$.
\end{proposition}

\begin{proof}
The convergence \eqref{fi} for $ \phi^{(0)}_\varepsilon$ immediately follows from the representation \eqref{Aw-t-r} for this function.
For the function  $ \phi^{(L)}_\varepsilon$, the proof is completely analogous to the proof of Proposition 5 in \cite{PZh}.
\end{proof}

\section{First corrector $\varkappa_1$ and drift $b$}

Our next step of the proof deals with constructing the correctors $\varkappa_1$ and $\varkappa_2$.
Denote $\xi=\frac{x}{\varepsilon}$ a variable on the period: $\xi \in \mathbb{T}^d = [0,1]^d$, then $ \mu(\xi,\eta), \varkappa^i_1(\xi), \varkappa^{ij}_2(\xi), \; i,j = 1, \ldots, d,$ are functions on $\mathbb T^d$.
We collect all the terms of the order $\varepsilon^{-1}$ in \eqref{Aw-t} and (\ref{K2_1}), and then equate them to 0. This yields the following equation for the vector function $\varkappa_1(\xi) = \{ \varkappa_1^i (\xi) \}, \ \xi \in \mathbb T^d, \; i=1, \ldots, d,$ as unknown function and for the unknown vector $b = \{ b^i \} \in \mathbb R^d$:
\begin{equation}\label{K1_2}
\int\limits_{\mathbb R^d}  \Big(  -z^i + \varkappa^i_1 (\xi-z) - \varkappa^i_1 (\xi) \Big) \ a (z) \mu (\xi, \xi -z ) \ dz +  b^i \ = \  0 \quad \forall \ i = 1, \ldots, d.
\end{equation}
Here and in what follows $\varkappa_1(q), \ q \in \mathbb R^d$, is the periodic extension of  $\varkappa_1(\xi), \ \xi \in \mathbb T^d$. Notice that (\ref{K1_2}) is a system of uncoupled equations.
After change of variables  $q=\xi-z \in \mathbb R^d$ equation (\ref{K1_2}) can be written in the vector form as follows
\begin{equation}\label{kappa_1bis}
\int\limits_{\mathbb R^d}  a (\xi-q) \mu (\xi, q)  (\varkappa_1 (q) - \varkappa_1 (\xi) ) \ dq \ = \ \int\limits_{\mathbb R^d}  a (\xi-q)  \mu (\xi, q ) (\xi - q ) \ dq \ - \  b,
\end{equation}
or
\begin{equation}\label{kappa_1}
A \varkappa_1 \ = \ h \ = \ f \ - \  b
\end{equation}
with the operator $A$ in $(L^2(\mathbb T^d))^d$ defined by
\begin{equation}\label{Akappa}
(A \bar\varphi) (\xi) \ = \ \int\limits_{\mathbb R^d}  a (\xi-q) \mu (\xi, q)  (\bar\varphi (q) - \bar\varphi (\xi) ) \ dq \ = \ \int\limits_{\mathbb T^d}  \hat a (\xi-\eta) \mu (\xi, \eta)  (\bar\varphi (\eta) - \bar\varphi (\xi) ) \ d\eta,
\end{equation}
where
\begin{equation}\label{hata}
\hat a(\eta) \ = \ \sum_{k \in \mathbb Z^d} a(\eta +k), \quad \eta \in \mathbb T^d,
\end{equation}
and
\begin{equation}\label{f}
 f \ = \ \int\limits_{\mathbb R^d}  a (\xi-q)  \mu (\xi, q ) (\xi - q ) \ dq.
\end{equation}
Observe that the vector function
\begin{equation}\label{fh}
h(\xi) \ = \ \int\limits_{\mathbb R^d}  a (\xi-q) \mu (\xi, q )  (\xi - q ) \ dq \ - \  b \ \in \ (L^2(\mathbb T^d))^d,
\end{equation}
because it is bounded for all $\xi \in \mathbb T^d$:
$$
\left| \int\limits_{\mathbb R^d}  a (\xi-q) (\xi - q ) \mu (\xi, q) \ dq \right| \ \le \ \alpha_2 \int\limits_{\mathbb R^d}  a (z) |z| \ dz \ < \ \infty.
$$
In \eqref{kappa_1} operator $A$ applies component-wise.  In what follows, abusing slightly the notation, we use the same notation $A$ for the scalar operator in $L^2(\mathbb T^d)$ acting on each component in \eqref{kappa_1}.

Let us denote
$$
K \varphi (\xi)=   \int\limits_{\mathbb R^d}  a (\xi-q) \mu (\xi, q) \varphi (q)  \, dq,\qquad
\varphi \in L^2(\mathbb T^d).
$$

\begin{proposition}[\cite{PZh}]\label{compact}
The operator
\begin{equation}\label{Kkappa}
K \varphi (\xi) \ = \  \int\limits_{\mathbb R^d}  a (\xi-q) \mu (\xi, q) \varphi (q)  \ dq \ = \   \int\limits_{\mathbb T^d}  \hat a (\xi-\eta) \mu (\xi, \eta) \varphi(\eta) \ d\eta, \quad  \varphi \in L^2(\mathbb T^d),
\end{equation}
is a compact operator in $L^2 (\mathbb T^d)$.
\end{proposition}

The proof see in \cite{PZh}.

\medskip

The operator
\begin{equation}\label{Gkappa}
G \varphi (\xi) \ = \  \varphi (\xi) \ \int\limits_{\mathbb R^d}  a (\xi-q) \mu (\xi, q)  \ dq \ = \ \varphi (\xi)  \int\limits_{\mathbb T^d}  \hat a (\xi-\eta) \mu (\xi, \eta)  \ d\eta, \quad \varphi \in L^2 (\mathbb T^d),
\end{equation}
is the operator of multiplication by the function $G(\xi) = \int\limits_{\mathbb R^d}  a (\xi-q) \mu (\xi, q)  \ dq$. Observe that
\begin{equation}\label{G}
0 < g_1 \le G(\xi) \le g_2 <\infty.
\end{equation}
Thus, the operator $A$ in (\ref{Akappa}) can be written as $A=K-G$, where $G$ and $K$ were defined in (\ref{Gkappa}) and (\ref{Kkappa}). Therefore $-A$ is the sum of a positive invertible operator $G$ and a compact operator $-K$, and the Fredholm theorem applies to (\ref{kappa_1}).

It will be shown in the next section that $\lambda=1$ is a simple eigenvalue of the operator $(G^{-1} K)^\ast$  in $L^2(\mathbb T^d)$.
Denote the corresponding eigenfunction by $\psi_0$.
It is easy to see that the kernel of $(G-K)^\ast$ has dimension one and that
\begin{equation}\label{Ker}
\mathrm{Ker} \, (G-K)^\ast = G^{-1}(\xi) \psi_0(\xi) =: v_0(\xi).
\end{equation}
Indeed,
$$
(G-K)^\ast v_0 = \big( G(E - G^{-1} K) \big)^\ast G^{-1} \psi_0 = \big( E - (G^{-1} K)^\ast \big) \psi_0 = 0.
$$
Then the solvability condition for the equation in  (\ref{kappa_1}) reads:
\begin{equation}\label{FA1}
\int\limits_{\mathbb T^d} h(\xi) v_0(\xi) \ d \xi \  = \ \int\limits_{\mathbb R^d} \int\limits_{\mathbb T^d}  a (\xi-q) \mu (\xi, q )  (\xi - q ) \ dq \, v_0 (\xi) d \xi \ - \  b \int\limits_{\mathbb T^d} \, v_0 (\xi) d \xi \ =  \ 0.
\end{equation}
Thus taking the normalized $v_0$ with $\int_{\mathbb T^d} v_0(\xi)d \xi =1$ and choosing $ b$ in the following way
\begin{equation}\label{FA1}
 b \ = \  \int\limits_{\mathbb R^d}  \int\limits_{\mathbb T^d} a (\xi-q) \mu (\xi, q )  (\xi - q ) \ dq \, v_0 (\xi) d \xi,
\end{equation}
we conclude that  the equation in \eqref{kappa_1bis} has a unique (up to a constant vector) solution $\varkappa_1 \in  (L^2 (\mathbb T^d))^d$.

The properties of the functions $\psi_0$ and $v_0$ are discussed in the next section.

\section{Ground state}

\begin{lemma}\label{GS}
The operator $(G^{-1} K)^\ast$ is compact in $L^2 (\mathbb T^d)$ and has a simple eigenvalue at $\lambda=1$.
The corresponding eigenfunction $\psi_0 \in L^2 (\mathbb T^d)$
satisfies the equation
\begin{equation}\label{GSequality}
(G^{-1} K)^\ast \psi_0 \ = \ \psi_0
\end{equation}
and admits the following estimates:
\begin{equation}\label{GSbound}
0< \gamma_1 \le \psi_0 (\xi) \le \gamma_2 < \infty \quad \mbox{\rm for all } \; \xi \in  \mathbb{T}^d,
\end{equation}
here  $\gamma_1>0$ and $\gamma_2$ are positive constants.
\end{lemma}

\begin{proof}
The compactness of $G^{-1}K$ is an immediate consequence of Proposition \ref{compact} and estimate  \eqref{G}.
The operator $A=K-G$ has an eigenfunction $\varphi_0(\xi) \equiv 1$ with the eigenvalue $\lambda=0$. Thus $\varphi_0(\xi) \equiv 1$ is also an eigenfunction of the operator $G^{-1}K$ that corresponds to the eigenvalue $\lambda=1$.
It is clear that  $G^{-1}K$ is a positive operator, that is it maps the set of non-negative $L^2(\mathbb T^d)$ functions
into itself.
Moreover, we will now prove that $G^{-1}K$ is a positivity improving operator, i.e. there exists $N \in \mathbb{N}$
such that
\begin{equation}\label{improving}
 f(\xi) \ge 0 \; \mbox{ implies } \;  (G^{-1}K)^N f(\xi)>0 \; \mbox{ for all } \; \xi \in \mathbb{T}^d.
\end{equation}
Due to representation \eqref{Kkappa} of the operator $K$ property \eqref{improving}  is a straightforward consequence of the following lemma.

\begin{lemma}\label{2conv}
There exist $N \in \mathbb{N}$ and $\gamma_0 >0$ such that
\begin{equation}\label{astN}
\hat a^{\ast N} (\xi) \ge \gamma_0 \quad \forall \xi \in \mathbb T^d,
\end{equation}
where the symbol $\ast$ stands for the convolution on the torus $\mathbb T^d$.
\end{lemma}

\begin{proof} For proving \eqref{astN} it is sufficient to show that
for any non-negative $a \in L^1 (\mathbb{R}^d)$:
\begin{equation}\label{L2conv-a}
a(z) \ge 0, \quad \int_{\mathbb{R}^d} a(z) \, dz =1,
\end{equation}
there exist $\gamma>0$ and a ball $B_\delta \in \mathbb{R}^d$ of a radius $\delta>0$ such that
\begin{equation}\label{L2conv}
(a \ast a)(z) > \gamma \quad \forall \; x \in B_\delta.
\end{equation}

The Lebesgue differentiation theorem states that, given any $f\in L^{1}(\mathbb{R}^{d})$, almost every $x$ is a Lebesgue point of $f$, i.e.
\begin{equation}\label{LP}
\lim _{{r\to 0^{+}}}{\frac  {1}{|B_r(x)|}}\int\limits_{{B_r(x)}}\!|f(y)-f(x)|\,d y = 0,
\end{equation}
where $B_r(x)$ is a ball centered at $x$  with radius $r>0$,  $|B_r(x)|$ is its Lebesgue measure.
Condition \eqref{L2conv-a} implies that there exists the Lebesgue point $x_0$ such that $a(x_0)= \alpha>0$. We assume without loss of generality that $x_0 = 0$.

\begin{proposition}
For any $\varepsilon>0$ there exists $\delta_0>0$ such that for any $\delta < \delta_0$:
\begin{equation}\label{Prop}
\mu \big\{ y \in B_\delta (0): \ a(y) > \frac{\alpha}{2} \big\} \ge (1 - \varepsilon) \, |B_\delta(0)|.
\end{equation}
\end{proposition}
\begin{proof}
Using inclusion
$$
\big\{  y \in B_\delta (0): \ a(y) < \frac{\alpha}{2} \big\} \subset \big\{ y \in B_\delta (0): \ |a(y) - a(0)| > \frac{\alpha}{2} \big\}, \quad \mbox{with } \ a(0)=\alpha,
$$
valid for any $\delta>0$, the Chebyshev inequality
$$
\mu \big\{ y \in B_\delta (0): \ |a(y) - a(0)| > \frac{\alpha}{2} \big\} \le \frac{2}{\alpha} \int\limits_{{B_\delta(0)}}\!|a(y)-a(0)|\,d y
$$
and definition \eqref{LP} of the Lebesgue point  we get that for any $\varepsilon>0$ there exists $\delta_0>0$ such that for any $\delta < \delta_0$:
\begin{equation}\label{Prop1}
\mu \big\{ y \in B_\delta (0): \ a(y) < \frac{\alpha}{2} \big\} \le \mu \big\{ y \in B_\delta (0): \ |a(y) - a(0)| > \frac{\alpha}{2} \big \} \le \varepsilon |B_\delta(0)|.
\end{equation}
Consequently, inequality \eqref{Prop} holds.
\end{proof}

Notice that  $x-y \in B_\delta(0)$, if $x,y \in B_{\frac{\delta}{2}}(0)$. Then
it follows from \eqref{Prop1} that for any $x \in B_{\frac{\delta}{2}}(0)$ we obtain
\begin{equation}\label{Lemma-1}
\mu \big\{ y \in B_{\frac{\delta}{2}} (0): \ a(y) > \frac{\alpha}{2}, \  a(x-y) > \frac{\alpha}{2} \big\} \ge
| B_{\frac{\delta}{2}}| - 2 \mu \big\{ y \in B_{\delta}(0): \ a(y) < \frac{\alpha}{2} \big\} \ge | B_{\frac{\delta}{2}}| - 2\varepsilon |B_\delta|.
\end{equation}
Choosing $\varepsilon = 2^{-(d+2)}$ and the corresponding $\delta = \delta(\varepsilon)$ we get from \eqref{Lemma-1} the following estimate which is valid for all $x \in B_{\frac{\delta}{2}}(0)$ with $\delta = \delta(\varepsilon)$:
\begin{equation}\label{Lemma-2}
\mu \big\{ y \in B_{\frac{\delta}{2}} (0): \ a(y) > \frac{\alpha}{2}, \  a(x-y) > \frac{\alpha}{2} \big\} \ge
\frac12 | B_{\frac{\delta}{2}}|.
\end{equation}
Finally we have for all $x \in B_{\frac{\delta}{2}}(0)$:
$$
(a \ast a) (x) = \int\limits_{\mathbb{R}^d} a(x-y) a(y) dy \ge \int\limits_{|y|< \frac{\delta}{2}, \ a(y) > \frac{\alpha}{2}, \ a(x-y) > \frac{\alpha}{2}} a(x-y) a(y) dy \ge \frac{\alpha^2}{8} | B_{\frac{\delta}{2}} |,
$$
which implies \eqref{L2conv}.  Since $\widehat{a^{\ast N}}(\cdot)=\hat a^{\star N}$, the inequality \eqref{astN} follows, and the proof of Lemma \ref{2conv} is completed.
\end{proof}

From Lemma \ref{2conv} it follows by the Krein-Rutman theorem (\cite{KR}, Theorem 6.2) that the operator $(G^{-1}K)^\ast$ has the maximal eigenvalue with the corresponding positive eigenfunction $\psi_0 >0$ (the ground state). As we have already noticed in the beginning of the proof of Lemma \ref{GS} the maximal eigenvalue of the operator $G^{-1}K$ is equal to 1. Consequently, the maximal eigenvalue of $(G^{-1}K)^\ast$ is also 1. The uniqueness of the ground state $\psi_0$ of the operator $(G^{-1}K)^\ast$ in the space $L^2({\mathbb T}^d)$ follows from the positivity improving property \eqref{improving}, see e.g.  \cite{KR}, Section 6. 

Thus we have proved the existence and uniqueness of $\psi_0>0, \ \psi_0 \in L^2({\mathbb T}^d)$ that satisfies  \eqref{GSequality}.  In particular,
$$
\|\psi_0\|_{L^1(\mathbb{T}^d)} = \int\limits_{\mathbb{T}^d} \psi_0(\xi) d \xi >0.
$$
Next we turn to the bounds in \eqref{GSbound}.
Estimates \eqref{G} and \eqref{astN} imply the bound from below:
\begin{equation}\label{B-below}
\psi_0 (\xi) = \big((G^{-1} K)^\ast \big)^N \psi_0 (\xi) \ge (g_2^{-1} \alpha_1)^N \gamma_0 \int\limits_{\mathbb{T}^d} \psi_0(\eta) d \eta = (g_2^{-1} \alpha_1)^N \gamma_0 \|\psi_0\|_{L^1(\mathbb{T}^d)} \quad \forall \; \xi \in \mathbb{T}^d,
\end{equation}
where $ 0<\|\psi_0\|_{L^1(\mathbb{T}^d)} \le  \|\psi_0\|_{L^2(\mathbb{T}^d)}$.
The upper bound follows from \eqref{M1} and \eqref{GSequality}:
\begin{equation}\label{B-upper}
\max\limits_{\xi}  \psi_0 (\xi) \le \max\limits_{\xi}  \Big| \int\limits_{\mathbb{T}^d} \hat a(\eta-\xi) \mu(\eta,\xi) G^{-1}(\eta)\psi_0(\eta) d \eta \Big| \le \alpha_2 g_1^{-1} \| \psi_0 \|_{L^2(\mathbb{T}^d)} \| \hat a \|_{L^2(\mathbb{T}^d)}.
\end{equation}
The proof of Lemma \ref{GS} is completed.
\end{proof} 

\begin{corollary}\label{Corolv0}
There exists a unique (up to an additive constant) function $v_0 \in L^2( \mathbb{T}^d)$ satisfying
\begin{equation}\label{v0}
\int\limits_{\mathbb R^d}  a (q-\xi) \mu (q, \xi) v_0 (q)  \ dq \ = \  v_0(\xi) \, \int\limits_{\mathbb R^d} a (\xi-q) \mu (\xi, q) dq, \end{equation}
i.e. $  \mathrm{span}(v_0)  =  \mathrm{Ker} \, (G-K)^\ast $.
This function obeys the following lower and upper bounds:
\begin{equation}\label{v0bound}
0< \tilde \gamma_1 \le v_0 (\xi) \le \tilde \gamma_2 < \infty \quad \mbox{ for all } \; \xi \in  \mathbb{T}^d.
\end{equation}
\end{corollary}

 \section{Second corrector $\varkappa_2$ and  effective matrix $\Theta$.}

We collect now all the terms of the order $\varepsilon^{0}$ in \eqref{Aw-t} and (\ref{K2_1}), and then equate them to the main term on the right-hand side of \eqref{mle}:
$$
\frac{\partial u}{\partial t}(x^\varepsilon,t) -
 \Theta \cdot \nabla \nabla u(x^\varepsilon,t).
$$
Notice that time derivatives $\frac{\partial u}{\partial t} (x^\varepsilon, t)$ are mutually cancelled on both sides of this relation, and we obtain an equation for the unknown matrix function $\varkappa_2(\xi) = \{ \varkappa_2^{ij} (\xi) \}, \ \xi \in \mathbb T^d, \; i,j=1, \ldots, d,$ and  the constant matrix $\Theta = \{ \Theta^{ij} \}$. This equation reads
\begin{equation}\label{K2}
\int\limits_{\mathbb R^d} a (z) \mu (\xi, \xi -z ) (\varkappa^{ij}_2 (\xi-z) - \varkappa^{ij}_2 (\xi) )  dz  +  b^i \varkappa_1^j(\xi)  +  \int\limits_{\mathbb R^d} a (z) \mu (\xi, \xi -z ) \big(\frac12 z^i z^j  - z^i \varkappa^{j}_1 (\xi-z)  \big)  dz\ = \  \Theta^{ij}.
\end{equation}
Notice that (\ref{K2}) is again a system of uncoupled equations.
After change of variables  $q=\xi-z \in \mathbb R^d$ equation (\ref{K2}) can be written in the vector form as follows
\begin{equation}\label{kappa_2bis}
\begin{array}{l}
\displaystyle
-\int\limits_{\mathbb R^d}  a (\xi-q) \mu (\xi, q)  (\varkappa_2 (q) - \varkappa_2 (\xi) ) \ dq
\\[4mm]  \displaystyle
=  \ b\otimes \varkappa_1(\xi) \ + \ \int\limits_{\mathbb R^d}  a (\xi-q)  \mu (\xi, q ) \Big( \frac12 (\xi - q )\otimes (\xi - q ) - (\xi-q)\otimes \varkappa_1 (q) \Big) dq \ - \ \Theta,
\end{array}
\end{equation}
or
\begin{equation}\label{kappa_2}
- A \varkappa_2 (\xi)\ = \ F(\xi) - \Theta
\end{equation}
with the operator $A$ defined above in \eqref {Akappa} and the following matrix function on the right-hand side:
$$
F(\xi) = \ b\otimes \varkappa_1(\xi) + \int\limits_{\mathbb R^d}  a (\xi-q)  \mu (\xi, q ) \Big( \frac12 (\xi - q )\otimes(\xi - q) - (\xi-q)\otimes \varkappa_1 (q) \Big) dq.
$$
The equation \eqref{kappa_2} on $\varkappa_2$ has the same form as equation \eqref{kappa_1} on  $\varkappa_1$. Consequently, using the same reasoning as above we conclude that the solvability condition for \eqref{kappa_2} leads after simple rearrangements to the following formula for the  matrix $\Theta$:
\begin{equation}\label{Theta}
\begin{array}{l}
\displaystyle
\Theta^{i j} \ = \  \int\limits_{\mathbb T^d} F^{i j}(\xi) v_0(\xi) \ d\xi
\\[7mm] \displaystyle
=  \  \int\limits_{\mathbb T^d}
\int\limits_{\mathbb R^d} a (\xi-q) \mu (\xi, q ) \Big( \frac12 (\xi - q )^i (\xi - q )^j  - (\xi - q )^i \varkappa_1^j (q) \Big) v_0(\xi) \ dq \  d\xi \ + \ b^i \int\limits_{\mathbb T^d} \varkappa_1^{j}(\xi) v_0(\xi) \ d\xi
 \end{array}
\end{equation}
for any $i, j$, where $v_0 \in L^2(\mathbb T^d)$  is the normalized function from $\mathrm{Ker}(-A^\ast)$, see \eqref{Ker}.

\begin{proposition}\label{l_2}
The integrals on the right-hand side of (\ref{Theta}) converge. Moreover, the symmetric part of the matrix
$\Theta = \{ \Theta^{i j} \}$ defined in (\ref{Theta}) is positive definite.
\end{proposition}

\begin{proof}
The first statement of the Proposition immediately follows from the existence of the second moment of the function $a(z)$.
Since function  $v_0(\xi)>0$ and satisfies two-sided bounds \eqref{v0bound}, it is sufficient to prove that the symmetric part  of the right-hand side of (\ref{Theta}) is positive definite.
To prove that $\Theta$ is a positive definite matrix we consider the following integrals:
\begin{equation}\label{I}
I^{ij}  =
\int\limits_{\mathbb T^d} \!  \int\limits_{\mathbb R^d}\!  a (\xi\!-\!q) \mu (\xi, q )  \big(  (\xi - q ) + ( \varkappa_1 (\xi) - \varkappa_1 (q)) \big)^i  \big(  (\xi - q ) + ( \varkappa_1 (\xi) - \varkappa_1 (q)) \big)^j v_0(\xi)  dq  d\xi.
\end{equation}
Our aim is to show that the symmetric part of the right-hand side of (\ref{Theta}) is equal to $I$:
\begin{equation}\label{Ibis}
I^{i j} \ = \ \Theta^{i j} + \Theta^{j i}.
\end{equation}
We have
\begin{equation}\label{ThetaSym}
\begin{array}{l}
\displaystyle
\Theta^{i j} \ +  \ \Theta^{j i}
\ =  \  \int\limits_{\mathbb T^d}
\int\limits_{\mathbb R^d} a (\xi-q) \mu (\xi, q )  (\xi - q )^i (\xi - q )^j v_0 (\xi) d\xi dq
\\[7mm] \displaystyle
 - \ \int\limits_{\mathbb T^d}
\int\limits_{\mathbb R^d} a (\xi-q) \mu (\xi, q ) \big( (\xi - q )^i \varkappa_1^j (q) + (\xi - q )^j \varkappa_1^i (q) \big) v_0(\xi) \ dq \ d\xi
\\[7mm] \displaystyle
+ \ b^i \int\limits_{\mathbb T^d} \varkappa_1^{j}(\xi) v_0(\xi) \ d\xi \ + \ b^j \int\limits_{\mathbb T^d} \varkappa_1^{i}(\xi) v_0(\xi) \ d\xi.
 \end{array}
\end{equation}
Let us rewrite $I^{ij}$ as the sum
$$
I^{ij} = I^{ij}_1 + I^{ij}_2 + I^{ij}_3,
$$
where
\begin{equation}\label{I1}
I^{ij}_1 \ = \
\int\limits_{\mathbb T^d} \!  \int\limits_{\mathbb R^d}\!  a (\xi\!-\!q) \mu (\xi, q )   (\xi - q )^i  (\xi - q )^j v_0(\xi)  dq  d\xi,
\end{equation}
\begin{equation}\label{I2}
I^{ij}_2 \ = \
\int\limits_{\mathbb T^d} \!  \int\limits_{\mathbb R^d}\!  a (\xi\!-\!q) \mu (\xi, q )  \big( (\xi - q )^i  ( \varkappa_1 (\xi) - \varkappa_1 (q))^j +  ( \varkappa_1 (\xi) - \varkappa_1 (q))^i (\xi - q )^j \big) v_0(\xi)  dq  d\xi,
\end{equation}
\begin{equation}\label{I3}
I^{ij}_3 \ = \
\int\limits_{\mathbb T^d} \!  \int\limits_{\mathbb R^d}\!  a (\xi\!-\!q) \mu (\xi, q )  ( \varkappa_1 (\xi) - \varkappa_1 (q))^i
( \varkappa_1 (\xi) - \varkappa_1 (q))^j  v_0(\xi)  dq  d\xi.
\end{equation}
Then $I^{ij}_1$ coincides with the first integral in \eqref{ThetaSym}. Let us rewrite the integral in  $I^{ij}_2$ as follows:
\begin{equation}\label{I2-1}
\begin{array}{l}
\displaystyle
I^{ij}_2 \ = \
\int\limits_{\mathbb T^d} \!  \int\limits_{\mathbb R^d}\!  a (\xi\!-\!q) \mu (\xi, q )  \big( (\xi - q )^i  \varkappa^j_1 (\xi) + (\xi - q )^j \varkappa^i_1 (\xi) \big) v_0(\xi)  dq  d\xi
\\[7mm] \displaystyle
- \ \int\limits_{\mathbb T^d} \!  \int\limits_{\mathbb R^d}\!  a (\xi\!-\!q) \mu (\xi, q )  \big( (\xi - q )^i  \varkappa^j_1 (q) + (\xi - q )^j \varkappa^i_1 (q) \big) v_0(\xi)  dq  d\xi \ = \ \tilde J^{ij}_2 \ + \  J^{ij}_2.
 \end{array}
\end{equation}
Then $ J^{ij}_2$ coincides with the second integral in \eqref{ThetaSym}. Further we rearrange the integral $\tilde J^{ij}_2$ using \eqref{kappa_1bis} and \eqref{kappa_1} and recalling the definition of the function $f$ in \eqref{f}:
\begin{equation}\label{J2}
\begin{array}{l}
\displaystyle
\tilde J^{ij}_2 \ = \
\int\limits_{\mathbb T^d} \!  f^i(\xi)  \varkappa^j_1 (\xi) v_0 (\xi) d\xi \  + \  \int\limits_{\mathbb T^d} \!  f^j(\xi)  \varkappa^i_1 (\xi) v_0 (\xi) d\xi
\\[7mm] \displaystyle
= \ \int\limits_{\mathbb T^d}   \varkappa^j_1 (\xi) v_0(\xi) \big(b^i + A \varkappa^i_1 (\xi)  \big) d\xi \ + \ \int\limits_{\mathbb T^d}   \varkappa^i_1 (\xi) v_0(\xi) \big(b^j + A \varkappa^j_1 (\xi)  \big) d\xi
\\[7mm] \displaystyle
= \ b^i \int\limits_{\mathbb T^d}   \varkappa^j_1 (\xi) v_0(\xi)  d\xi \ + \ b^j \int\limits_{\mathbb T^d}   \varkappa^i_1 (\xi) v_0(\xi)  d\xi
\\[7mm] \displaystyle
 + \  \int\limits_{\mathbb T^d}   \varkappa^j_1 (\xi) v_0(\xi)  A \varkappa^i_1 (\xi)   d\xi \ + \ \int\limits_{\mathbb T^d}   \varkappa^i_1 (\xi) v_0(\xi)  A \varkappa^j_1 (\xi)   d\xi.
 \end{array}
\end{equation}
Denote
\begin{equation}\label{D2}
D^{ij}_2 = \ b^i \int\limits_{\mathbb T^d}   \varkappa^j_1 (\xi) v_0(\xi)  d\xi \ + \ b^j \int\limits_{\mathbb T^d}   \varkappa^i_1 (\xi) v_0(\xi)  d\xi,
\end{equation}
\begin{equation}\label{D2bis}
\tilde D^{ij}_2 = \int\limits_{\mathbb T^d}   \varkappa^j_1 (\xi) v_0(\xi)  A \varkappa^i_1 (\xi)   d\xi \ + \ \int\limits_{\mathbb T^d}   \varkappa^i_1 (\xi) v_0(\xi)  A \varkappa^j_1 (\xi)   d\xi.
\end{equation}
Then $D^{ij}_2$ coincides with the third integral in \eqref{ThetaSym}.

We have to show that $I^{ij}_3 = - \tilde D^{ij}_2$. We have
\begin{equation}\label{I3-1}
\begin{array}{l}
\displaystyle
I^{ij}_3 \ = \
\int\limits_{\mathbb T^d} \!  \int\limits_{\mathbb R^d}\!  a (\xi\!-\!q) \mu (\xi, q )  ( \varkappa_1 (\xi) - \varkappa_1 (q))^i \varkappa^j_1 (\xi)  v_0(\xi)  dq  d\xi
\\[7mm] \displaystyle
- \ \int\limits_{\mathbb T^d} \!  \int\limits_{\mathbb R^d}\!  a (\xi\!-\!q) \mu (\xi, q )  ( \varkappa_1 (\xi) - \varkappa_1 (q))^i \varkappa^j_1 (q)  v_0(\xi)  dq  d\xi
\\[7mm] \displaystyle
= \ - \int\limits_{\mathbb T^d}   A \varkappa^i_1 (\xi) \varkappa^j_1 (\xi) v_0(\xi)    d\xi \ + \  J^{ij}_3.
 \end{array}
\end{equation}
We rearrange $J^{ij}_3$ using \eqref{v0}:
\begin{equation}\label{I3-2}
\begin{array}{l}
\displaystyle
J^{ij}_3 \ = \
- \int\limits_{\mathbb T^d} \!  \int\limits_{\mathbb R^d}\!  a (\xi\!-\!q) \mu (\xi, q )  ( \varkappa_1 (\xi) - \varkappa_1 (q))^i \varkappa^j_1 (q)  v_0(\xi)  dq  d\xi
\\[7mm] \displaystyle
= \  \int\limits_{\mathbb T^d}  \int\limits_{\mathbb T^d}\!  \hat a (\xi\!-\!q) \mu (\xi,q )  ( \varkappa_1 (q) - \varkappa_1 (\xi))^i \varkappa^j_1 (q)  v_0(\xi)  dq  d\xi
\\[7mm] \displaystyle
= \  \int\limits_{\mathbb T^d}  \int\limits_{\mathbb T^d}\! \hat a (q\!-\xi\!) \mu (q,\xi )  ( \varkappa_1 (\xi) - \varkappa_1 (q))^i \varkappa^j_1 (\xi)  v_0(q)  dq  d\xi
\\[7mm] \displaystyle
= \  \int\limits_{\mathbb T^d}  \int\limits_{\mathbb T^d}\! \hat a (q\!-\xi\!) \mu (q,\xi ) v_0(q) dq \ \varkappa^i_1 (\xi) \varkappa^j_1 (\xi)    d\xi -  \int\limits_{\mathbb T^d}  \int\limits_{\mathbb T^d}\!  a (q\!-\xi\!) \mu (q,\xi ) v_0(q) \varkappa_1 (q)^i \varkappa^j_1 (\xi) \ dq   d\xi
\\[7mm] \displaystyle
= \  \int\limits_{\mathbb T^d}  \int\limits_{\mathbb T^d}\! \hat a (\xi\!-\!q) \mu (\xi,q ) \varkappa^j_1 (\xi)  dq \ v_0(\xi) \varkappa^i_1 (\xi)    d\xi -  \int\limits_{\mathbb T^d}  \int\limits_{\mathbb T^d}\! \hat a (\xi\!-\!q) \mu (\xi,q ) \varkappa^j_1 (q) dq \ v_0(\xi) \varkappa^i_1 (\xi)   d\xi
\\[7mm] \displaystyle
= \ - \int\limits_{\mathbb T^d}   A \varkappa^j_1 (\xi) v_0(\xi)  \varkappa^i_1 (\xi)  d\xi.
 \end{array}
\end{equation}
Thus $I^{ij}_3 = - \tilde D^{ij}_2$ and this relation complete the proof of equality \eqref{Ibis}.

The structure of (\ref{I}) implies that $(Ir,r) \ge 0, \; \forall \ r \in \mathbb R^d$, and moreover $(Ir,r)>0$ since $v_0>0$ and $\varkappa_1(q)$ is the periodic function while $q$ is the linear function, consequently
$\big[ \big( (\xi - q ) + ( \varkappa_1 (\xi) - \varkappa_1 (q))\big)\cdot r \big]^2 $ can not be identically 0 if $r\not=0$.
\end{proof}

Thus, the Lemma \ref{ml} is now completely proved.
\end{proof}

\section{A priori estimates}

Let  $u^0(x,t)$ be a solution of \eqref{PE} with $u^0(x,0) = \varphi \in {\cal{S}}(\mathbb{R}^d)$. Then $ u^0(x,t) \in C^{\infty}((0,T),{\cal{S}}(\mathbb R^d))$ for any $T$ and we can define approximation $w^\varepsilon$ of $u^0$ substituting
$u^0(\cdot)$ for $u(\cdot)$ in \eqref{w_eps}. It follows from Lemma \ref{ml}  that $w^\varepsilon$ satisfies the following equation
\begin{equation}\label{w_eps_bis}
\frac{\partial w^{\varepsilon}}{\partial t} -  L^{\varepsilon} w^{\varepsilon} =
\frac{\partial u^0}{\partial t}(x^\varepsilon,t) -
 \Theta \cdot \nabla \nabla u^0(x^\varepsilon,t) \ + \ \phi^\varepsilon(x^\varepsilon,t) =
\phi^\varepsilon(x^\varepsilon,t), \quad w^{\varepsilon}(x,0) = \varphi(x) + \psi^\varepsilon(x)
\end{equation}
where $x^\varepsilon = x - \frac{b}{\varepsilon}\, t$, and
$$
\psi^\varepsilon (x) \ = \ \varepsilon \varkappa_1 (\frac{x}{\varepsilon})\cdot \nabla \varphi(x) + \varepsilon^2 \varkappa_2 (\frac{x}{\varepsilon})\cdot \nabla \nabla \varphi (x) \ \in \ L^2(\mathbb{R}^d).
$$
Consequently, the difference  $v^\varepsilon(x,t) = w^\varepsilon(x,t) - u^\varepsilon(x,t)$, where $u^\varepsilon$ is the solution of \eqref{PE_eps}, satisfies the following problem:
\begin{equation}\label{v_eps_bis}
\frac{\partial v^{\varepsilon}(x,t)}{\partial t} -  L^{\varepsilon} v^{\varepsilon}(x,t) = \phi^\varepsilon(x^\varepsilon,t), \quad v^{\varepsilon}(x,0) =  \psi^\varepsilon(x).
\end{equation}
Notice that by \eqref{fiplus} and Proposition \ref{fi_0} we have $\| \psi^\varepsilon \|_{L^2(\mathbb{R}^d)} = O(\varepsilon)$ and $\| \phi^\varepsilon \|_{\infty} = o(1)$, where $\|\cdot\|_{\infty} $ is the norm in $L^{\infty}((0,T),L^2(\mathbb R^d))$.
We are going to show now that the solution $v^\varepsilon$ of \eqref{v_eps_bis} tends to zero in $L^{\infty}((0,T),L^2(\mathbb R^d))$ as
$\eps\to0$.

\begin{proposition}\label{Proposition_v}
 Let $v^\varepsilon$ be the solution of \eqref{v_eps_bis} with  small $\psi^\varepsilon$ and $\phi^\varepsilon$:
$$
\| \phi^\varepsilon \|_{\infty} = o(1), \quad \| \psi^\varepsilon \|_{L^2(\mathbb{R}^d)} = O(\varepsilon) \quad \mbox{as } \; \varepsilon \to 0.
$$
Then
\begin{equation}\label{v_epsto0}
\| v^\varepsilon \|_{\infty} \ \to \ 0 \quad \mbox{as } \; \varepsilon \to 0.
\end{equation}
\end{proposition}
\begin{proof}
Since problem \eqref{v_eps_bis} is linear, we consider separately two problems:
\begin{equation}\label{v_1bis}
\frac{\partial v_\psi}{\partial t} -  L^{\varepsilon} v_\psi = 0, \quad v_\psi(x,0) =  \psi(x),
\end{equation}
\begin{equation}\label{v_2bis}
\frac{\partial v_\phi}{\partial t} -  L^{\varepsilon} v_\phi = \phi, \quad v_\phi(x,0) =  0,
\end{equation}
and prove that $\| v_\psi \|_{\infty} \le C_1 \| \psi \|_{L^2(\mathbb{R}^d)}$ 
 and
 $\| v_\phi \|^2_{\infty} \le C_2 \| \phi \|_{\infty}$
with some constants $C_1,\, C_2$ that do not depend on $\eps$, however might depend on $T$.
This immediately implies the required relation in \eqref{v_epsto0}.

Denote $v_0^\varepsilon(x) = \tilde v_0(\frac{x}{\varepsilon})$, where $\tilde v_0$ is the periodic extension of the function $v_0 \in L^2(\mathbb{T}^d)$ defined in \eqref{Ker}, see also Corollary \ref{Corolv0}. Multiplying equation \eqref{v_1bis} by $ v_\psi(x,t) \, \tilde v_0(\frac{x}{\varepsilon})$ and integrating the resulting relation  over $t \in (0,s)$ and $x \in \mathbb{R}^d$ we have
\begin{equation}\label{apriori1}
\begin{array}{l}
\displaystyle
\int\limits_{\mathbb{R}^d} \int\limits_0^s \frac{\partial v_\psi(x,t)}{\partial t} \, v_\psi(x,t) \, dt \, v_0^\varepsilon(x) \, dx = \frac12  \int\limits_{\mathbb{R}^d}  v_\psi^2 (x,s) \, v_0^\varepsilon(x) \, dx  -  \frac12  \int\limits_{\mathbb{R}^d} \psi^2 (x) \,  \, v_0^\varepsilon(x) \, dx
\\ [7mm] \displaystyle
=  \int\limits_0^s \int\limits_{\mathbb{R}^d} v_\psi(x,t) \, v_0^\varepsilon(x) \, L^\varepsilon v_\psi (x,t)\, dx \, dt.
\end{array}
\end{equation}
All integrals in \eqref{apriori1} exist since $v_0$ is uniformly bounded, see \eqref{v0bound}. The last integral in \eqref{apriori1} can be analysed in the same way as the term $I_3$ in the proof of Proposition \ref{l_2}, 
see \eqref{I3-1} - \eqref{I3-2}. This yields
$$
\int\limits_{\mathbb{R}^d} L^\varepsilon v_\psi(x,t) \, v_\psi(x,t) \, v_0^\varepsilon(x) \, dx \ = \  (L^\varepsilon v_\psi, v_\psi)_{v_0} \ \le \ 0 \quad \mbox{for all } \; t \in [0, T],
$$
and consequently,
$$
\int\limits_{\mathbb{R}^d}  v_\psi^2 (x,s) \,  v_0^\varepsilon(x) \, dx \ \le \  \int\limits_{\mathbb{R}^d} \psi^2 (x) \,  \, v_0^\varepsilon(x) \, dx \quad \mbox{ for all } \; s \in (0,T).
$$
Using the estimates in \eqref{v0bound} for $v_0$ we conclude that
\begin{equation}\label{apriori1bis1}
\| v_\psi (\cdot, s) \|_{L^2(\mathbb{R}^d)} \ \le \ C_1 \| \psi \|_{L^2(\mathbb{R}^d)}
\end{equation}
with a constant $C_1$ which does not depend on $s \in (0,T)$. Thus
\begin{equation}\label{apriori1bis}
\| v_\psi \|_{\infty} \ \le \ C_1 \| \psi \|_{L^2(\mathbb{R}^d)}.
\end{equation}
\medskip

Using the same reasoning for the second equation \eqref{v_2bis} we obtain
\begin{equation}\label{apriori2}
\frac12 \int\limits_{\mathbb{R}^d}  v_\phi^2 (x,s) \,  v_0^\varepsilon(x) \, dx \ - \ \int\limits_0^s \int\limits_{\mathbb{R}^d} \phi (x,t) \,  v_\phi(x,t) \, v_0^\varepsilon(x) \, dx \, dt \ = \   \int\limits_0^s  (L^\varepsilon v_\phi, v_\phi)_{v_0} \, dt   \ \le \ 0.
\end{equation}
Recalling the bounds in \eqref{v0bound}, by the Schwartz inequality we derive from \eqref{apriori2} that 
\begin{equation}\label{apriori2-1}
\frac{\tilde \gamma_1}{2} \| v_\phi (\cdot, s) \|^2_{L^2(\mathbb{R}^d)} \ \le \ \frac{\tilde \gamma_2}{2}  \int\limits_0^s \| \phi (\cdot,t) \|_{L^2(\mathbb{R}^d)} \, \| v_\phi (\cdot,t) \|_{L^2(\mathbb{R}^d)} \, dt \ \le
\ \frac{\tilde \gamma_2}{2} \, s \, \| \phi \|_{\infty} \, \| v_\phi  \|_{\infty}
\end{equation}
for any $s \in (0,T)$. 
Consequently,
$$
\| v_\phi \|_{\infty} \le C_2(T) \, \| \phi \|_{\infty}.
$$
\end{proof}

\medskip

Since $\| w^\varepsilon (x,t) - u^0 (x - \frac{b}{\varepsilon} \, t, t) \|_{\infty} \to 0$  by \eqref{w_eps}, then \eqref{v_epsto0} immediately yields
 \begin{equation}\label{u_epsto0}
\| u^\varepsilon (x,t) \ - \ u^0 (x - \frac{b}{\varepsilon} \, t, t) \|_{\infty} \ \to \ 0 \quad \mbox{or } \; \| u^\varepsilon (x + \frac{b}{\varepsilon} \, t,t) \ - \ u^0 (x, t) \|_{\infty} \ \to \ 0  \quad  \mbox{as } \; \varepsilon \to 0.
\end{equation}
Thus we proved \eqref{T1} for a dense in $L^2(\mathbb{R}^d)$ set of initial data, when $\varphi \in {\cal{S}}(\mathbb R^d)$.

\medskip
We can  complete now the proof of Theorem \ref{Theorem1}. For any $\varphi \in L^2(\mathbb R^d)$ and for any $\delta>0$ there exists $ \varphi_\delta \in \mathcal{S}(\mathbb R^d)$ such that $\| \varphi - \varphi_\delta\|_{L^2(\mathbb R^d)} <\delta$.
We denote by  $ u^\varepsilon_{\delta} $ and $ u^0_{\delta} $ the solution of \eqref{PE_eps} and \eqref{PE} with initial data $\varphi_\delta$.
Since \eqref{PE} is the standard Cauchy problem for a parabolic operator with constant coefficients, its
solution admits the classical upper bound
\begin{equation}\label{F1}
 \| u^0 (x,t)\ - \ u^0_{\delta}(x,t) \|_{\infty} \ \le \|\varphi-\varphi_\delta\|_{L^2(\mathbb R^d)} \ < \delta
\end{equation}
for any $T>0$.
By the estimate in \eqref{apriori1bis1} we obtain
\begin{equation}\label{F2}
\| u^\varepsilon_\delta (x,t) \ - \ u^\varepsilon (x, t) \|_{\infty} \ \le \ C_1 \, \delta.
\end{equation}
Since the upper bounds in \eqref{F1} - \eqref{F2} are valid with an arbitrary small $\delta>0$, then \eqref{u_epsto0} - \eqref{F2} 
imply that
$$\| u^{\varepsilon}  (x + \frac{b}{\varepsilon} \, t,t) - u^0 (x,t) \|_{\infty} \to 0, \quad \mbox{as } \; \varepsilon \to 0.
$$
This completes the proof of Theorem \ref{Theorem1}.

\section{Small perturbations of symmetric kernels. Einstein relation.}

Let us assume in this section that $\mu(\xi,\eta) = \mu(\eta,\xi)$ and consider a kernel $a(z)$ satisfying \eqref{M1} - \eqref{M2} of a special form:
\begin{equation}\label{aer}
a(z) \ = \ a_{\rm sym}(z) \ + \ \boldsymbol{\ell} \cdot c(z),
\end{equation}
where $ a_{\rm sym}(-z) = a_{\rm sym}(z) $ is a symmetric function that also satisfies \eqref{M1} - \eqref{M2}, $c(z)$
is an antisymmetric vector function, that is
 $\boldsymbol{\ell} \cdot c(z) = \ell^i c^i(z)$, $c^i(-z)= - c^i(z), \, i= 1, \ldots, d$;  $c(z)$ satisfies condition   \eqref{M2},
and $\boldsymbol{\ell} \in \mathbb{R}^d$ is a constant vector of a small norm. We assume here and in the sequel summation over repeated indices.
We also consider in this section a special case of antisymmetric perturbation of the form
$$
c_{\boldsymbol{\ell}}(z)=z a_{\rm sym}(z)\omega_{\boldsymbol{\ell}}(z),
$$
where $\omega_{\boldsymbol{\ell}}(z)=\omega(|\boldsymbol{\ell}|\, |z|)$, and $\omega(s)$ is a $C_0^\infty(\mathbb R)$ function such that
$0\leq\omega(\cdot)\leq 1$, $\omega(s)=1$ for $s\in [0,\frac14]$, and  $\omega(s)=0$ for $s\geq \frac12$.
\begin{lemma}\label{ER}
Let $b(\boldsymbol{\ell}) \in \mathbb R^d$ be the effective drift vector corresponding to the problem \eqref{K1_2} with $a(z)$ given by \eqref{aer}.
Then, for small ${\boldsymbol{\ell}}$,
\begin{equation}\label{Blambda}
b^i({\boldsymbol{\ell}}) \ = \ \ell^j \, \int\limits_{\mathbb R^d}  \int\limits_{\mathbb T^d} z^i \, c^j(z) \, \mu (\xi, \xi-z ) dz d\xi \ + \ \ell^j  \int\limits_{\mathbb R^d}  \int\limits_{\mathbb T^d}
 z^i \,  a_{\rm sym} (z) \mu(\xi,\xi-z) \,  \tilde\varphi^j_{0}(\xi)   dz d\xi + O(|{\boldsymbol{\ell}}|^2),
\end{equation}
where $\tilde\varphi_0 = \{ \tilde\varphi^i_{0} \}  \in (L^2 (\mathbb T^d))^d$ is the solution of the problem
\begin{equation}\label{ER4bis}
\int\limits_{\mathbb{R}^d} a_{\rm sym} (\xi-q) \mu(\xi,q) \big( \tilde\varphi^i_{0} (q) -  \tilde\varphi^i_{0} (\xi) \big)  \, d q \ = \ 2 \, \int\limits_{\mathbb{R}^d} c^i(\xi-q) \mu(\xi,q) \, dq
\end{equation}
with  $\int\limits_{\mathbb{T}^d} \tilde\varphi^i_{0}(\eta)d\eta = 0$.

In the special case, when $c_{\boldsymbol{\ell}} (z) \ = \ z \,  a_{\rm sym}(z) \, \omega_{\boldsymbol{\ell}} (z)$  and  $b({\boldsymbol{\ell}})$ is defined by \eqref{Blambda} - \eqref{ER4bis} with $c(z) = c_{\boldsymbol{\ell}} (z)$,  we obtain the so-called Einstein relation:
\begin{equation}\label{BER}
\frac{\partial b^i (\boldsymbol{\ell})}{\partial {\ell}^j}\Big|_{\boldsymbol{\ell} = 0} \ = \  2 \Theta_{\rm sym}^{ij},
\end{equation}
where $\Theta_{\rm sym}$ is the effective matrix of  problem \eqref{PE_eps} corresponding to the symmetric kernel
$a_{\rm sym}(x-y) \mu(x,y)$.
\end{lemma}

\begin{remark}
Notice that the symmetric part of $2 \Theta_{\rm sym}^{ij}$ coincides with $I_{\rm sym}^{ij} = \Theta_{\rm sym}^{ij} + \Theta_{\rm sym}^{ji}$.
\end{remark}

\begin{proof}  Since the operator $K$ and the function $G$ defined in \eqref{Kkappa} and \eqref{Gkappa}, respectively,
depend on a vector parameter ${\boldsymbol{\ell}}$ smoothly, and $\lambda=1$ is a simple eigenvalue of the operator $((G(\cdot))^{-1}K)^\ast$,
then  the corresponding eigenfunction $\psi_0=\psi_0^{\boldsymbol{\ell}}\in L^2(\mathbb T^d)$ is also a smooth function of a parameter ${\boldsymbol{\ell}}$.
So is  $v_0=v_0^{\boldsymbol{\ell}}$.
Using the perturbation theory arguments we conclude that for small ${\boldsymbol{\ell}}$ the function $v_0^{\boldsymbol{\ell}} \in L^2 (\mathbb T^d)$ defined by \eqref{Ker} admits the following representation
\begin{equation}\label{vl}
v_0^{\boldsymbol{\ell}} (\xi) \ = \ \mathbf{1} \ + \ {\boldsymbol{\ell}} \, \tilde \varphi_0 (\xi) \ + \ O({|\boldsymbol{\ell}|}^2), \quad  \tilde \varphi_0  \in (L^2 (\mathbb T^d))^d,
\end{equation}
where ${\bf 1}$ stands for the function identically equal to $1$ on $\mathbb T^d$.
We used here the fact that
\begin{equation}\label{ER1}
 \mathrm{span}( \mathbf{1} ) = \mathrm{Ker} \, (G_{\rm sym} - K_{\rm sym})^\ast = \mathrm{Ker} \, (G_{\rm sym} - K_{\rm sym}),
\end{equation}
where operators $K$, $G$ are defined by \eqref{Kkappa} and \eqref{Gkappa} respectively, and we denote by $K_{\rm sym}$, $G_{\rm sym}$ the  operators related to the symmetric kernel $a_{\rm sym}(x-y)\mu(x,y)$.

Substituting \eqref{vl} in  the relation $K^\ast v^{\boldsymbol{\ell}}_0 \ = \ G v^{\boldsymbol{\ell}}_0$  we obtain
\begin{equation}\label{ER2}
\begin{array}{l}
\displaystyle
\int\limits_{\mathbb{R}^d} a_{\rm sym} (q-\xi) \mu(q,\xi) \big( \mathbf{1}+ \ell_i \tilde\varphi^i_{0}(q) \big) \, dq \ + \  \ell^j \int\limits_{\mathbb{R}^d} c^j (q-\xi) \mu(q,\xi) \big( \mathbf{1} + \ell^i \tilde\varphi^i_{0}(q) \big) \, dq \ + \ O({|\boldsymbol{\ell}|}^2)
\\ [7mm] \displaystyle
= \ \big( \mathbf{1}+ \ell^i \tilde\varphi^i_{0}(\xi) \big) \Big[ \int\limits_{\mathbb{R}^d} a_{\rm sym} (\xi-q) \mu(\xi,q) \, dq \ + \  \ell^j \int\limits_{\mathbb{R}^d} c^j(\xi-q) \mu(\xi,q) \, dq \Big] \ + \ O({|\boldsymbol{\ell}|}^2).
\end{array}
\end{equation}
Relations \eqref{ER1} - \eqref{ER2} yield
\begin{equation}\label{ER3}
\begin{array}{l}
\displaystyle
\ell^i \, \int\limits_{\mathbb{R}^d} a_{\rm sym} (\xi-q) \mu(\xi,q) \tilde\varphi^i_{0}(q)  \, dq \ - \  \ell^i \, \int\limits_{\mathbb{R}^d} c^i(\xi-q) \mu(\xi,q) \, dq \ + \ O({|\boldsymbol{\ell}|}^2)
\\ [7mm] \displaystyle
= \  \ell^i \, \tilde\varphi^i_{0}(\xi) \, \int\limits_{\mathbb{R}^d} a_{\rm sym} (\xi-q) \mu(\xi,q) \, dq \ + \  \ell^i \, \int\limits_{\mathbb{R}^d} c^i(\xi-q) \mu(\xi,q) \, dq \  + \ O({|\boldsymbol{\ell}|}^2).
\end{array}
\end{equation}
Collecting the terms of the order $|\boldsymbol{\ell}|$ in \eqref{ER3} we deduce the equation for $\tilde\varphi^i_{0}$: 
\begin{equation}\label{ER4}
\int\limits_{\mathbb{R}^d} a_{\rm sym} (\xi-q) \mu(\xi,q) \big( \tilde\varphi^i_{0}(q) -  \tilde\varphi^i_{0}(\xi) \big)  \, dq \ = \ 2 \, \int\limits_{\mathbb{R}^d} c^i(\xi-q) \mu(\xi,q) \, dq.
\end{equation}
\medskip

Our subsequent reasoning relies on the following statement.

\begin{proposition}\label{p_from_6}
If $\alpha (-z) = \alpha(z)$ for all $z\in\mathbb R^d$ and $\alpha\in L^1(\mathbb R^d)$, then
\begin{equation}\label{L1_bis}
\int\limits_{\mathbb R^d} \int\limits_{\mathbb T^d}  \alpha (\xi-q) \mu (\xi,q)  \, d\xi  dq = \int\limits_{\mathbb R^d} \int\limits_{\mathbb T^d}  \alpha (\xi-q) \mu (q,\xi)  \, d\xi  dq;
\end{equation}
if $\beta (-z)= - \beta (z)$ for all $z\in\mathbb R^d$ and $\beta\in L^1(\mathbb R^d)$, then
\begin{equation}\label{L1_bis}
\int\limits_{\mathbb R^d} \int\limits_{\mathbb T^d}  \beta (\xi-q) \mu (\xi,q)  \, d\xi  dq = - \int\limits_{\mathbb R^d} \int\limits_{\mathbb T^d}  \beta (\xi-q) \mu (q,\xi)  \ d\xi \ dq.
\end{equation}
\end{proposition}
\begin{proof}[The proof {\rm is the same as that of Proposition 7 in \cite{PZh}}]
It is straightforward to check that the arguments used in the proof given in  \cite{PZh}  also apply to the operators
considered here. We leave the details to the reader.
\end{proof}

Since $c^i (-z) = -c^i (z)$ by our assumption, then Proposition \ref{p_from_6} yields
$$
\int\limits_{\mathbb{R}^d} \int\limits_{\mathbb{T}^d} c^i(\xi-q) \mu(\xi,q) \, dq  d\xi \ = \ 0,
$$
and consequently, there exists a unique (up to an additive constant) solution $ \tilde\varphi_0 \in (L^2(\mathbb{T}^d))^d$ of \eqref{ER4}. We choose the additive constant in such a way that $\int\limits_{\mathbb{T}^d} \tilde\varphi^i_{0}(\xi) d\xi = 0$ for any component of $\tilde\varphi_{0}$. Then \eqref{vl} implies that $\int\limits_{\mathbb{T}^d} v^{\boldsymbol{\ell}}_0(\xi)d\xi = 1 + O({|\boldsymbol{\ell}|}^2)$, and from \eqref{FA1} and \eqref{vl} we obtain that
\begin{equation}\label{ER4biss}
\begin{array}{l}
\displaystyle
 b^i({\boldsymbol{\ell}}) \ = \  \int\limits_{\mathbb R^d}  \int\limits_{\mathbb T^d} z^i \, \big(a_{\rm sym} (z) + \ell^j c^j (z)\big) \, \mu (\xi, \xi-z ) \, dz \,  \big(\mathbf{1} + \ell^j \tilde\varphi^j_{0} (\xi) \big) \, d\xi \ + \ O({|\boldsymbol{\ell}|}^2) \\ [7mm] \displaystyle
= \  \ell^j \, \int\limits_{\mathbb R^d}  \int\limits_{\mathbb T^d} z^i \, c^j(z) \, \mu (\xi, \xi-z )\, dz d\xi \ + \ \ell^j \, \int\limits_{\mathbb R^d}  \int\limits_{\mathbb T^d}
 z^i \,  a_{\rm sym} (z) \, \mu(\xi,\xi-z) \,  \tilde\varphi^j_{0}(\xi) \,  dz d\xi + O({|\boldsymbol{\ell}|}^2).
\end{array}
\end{equation}

\bigskip

In the case when $c_{\boldsymbol{\ell}} (z) \ = \ z \,  a_{\rm sym}(z) \, \omega_{\boldsymbol{\ell}}(z)$, it follows from equation \eqref{ER4} that \begin{equation}\label{eps1}
\tilde\varphi_0^{\boldsymbol{\ell}} = 2 \varkappa_{\rm sym} + r_{\boldsymbol{\ell}}, \quad \tilde\varphi_0^{\boldsymbol{\ell}}  \in (L^2 (\mathbb T^d))^d,
\end{equation}
where $\varkappa_{\rm sym}$ is the first corrector of the symmetric problem \eqref{PE_eps} that  satisfies the equation
\begin{equation}\label{ER4sym}
\int\limits_{\mathbb{R}^d} a_{\rm sym} (\xi-q) \, \mu(\xi,q) \, \big( \varkappa_{\rm sym}(q) -  \varkappa_{\rm sym}(\xi) \big)  \, dq \ = \ \int\limits_{\mathbb{R}^d} a_{\rm sym} (\xi-q) \, (\xi-q) \, \mu(\xi,q) \, dq,
\end{equation}
see also \cite{PZh}, and
\begin{equation}\label{vareps}
\| r_{\boldsymbol{\ell}} \|_{(L^2(\mathbb{T}^d))^d} \ \to \ 0 \quad \mbox{ as } \;{\boldsymbol{\ell}} \to 0.
\end{equation}
Indeed, denoting
$$
g_{\boldsymbol{\ell}}(\xi) = \int\limits_{\mathbb R^d}  z \,  a_{\rm sym} (z) \, \big(1 -   \omega_{\boldsymbol{\ell}}(z) \big)  \, \mu(\xi,\xi-z) \,  dz  
$$
and using \eqref{M2} we get that
\begin{equation}\label{eps2}
\max_{i=1,\ldots, d} \ \max_{\xi \in \mathbb{T}^d} |g^i_{\boldsymbol{\ell}}(\xi) | \ \to \ 0 \;\; \mbox{as } \; {\boldsymbol{\ell}} \to 0, \quad \mbox{and consequently } \; \| g_{\boldsymbol{\ell}} \|_{(L^2(\mathbb{T}^d))^d} \ \to \ 0 \;\; \mbox{ as } \;  {\boldsymbol{\ell}} \to 0.
\end{equation}
After substitution \eqref{eps1} in \eqref{ER4}, considering equation \eqref{ER4sym}, 
we come to a conclusion that $r_{\boldsymbol{\ell}}$ satisfies the following equation:
\begin{equation}\label{eps3}
\int\limits_{\mathbb T^d}  \hat a_{\rm sym} (\xi-\eta) \, \mu(\xi,\eta) \, (r_{\boldsymbol{\ell}}(\eta) - r_{\boldsymbol{\ell}}(\xi)) d\eta \ = \ 2 g_{\boldsymbol{\ell}}(\xi), \quad  r_{\boldsymbol{\ell}} \in (L^2(\mathbb{T}^d))^d.
\end{equation}
We can rewrite \eqref{eps3} as
\begin{equation}\label{eps4}
\Big(K_{\rm sym} \ - \ G_{\rm sym}  \Big) r_{\boldsymbol{\ell}} \ = \ 2 g_{\boldsymbol{\ell}},
\end{equation}
where the operators
$$
K_{\rm sym}f(\xi) =  \int\limits_{\mathbb T^d}  \hat a_{\rm sym} (\xi-\eta) \, \mu(\xi,\eta) \, f(\eta) d\eta, \;
G_{\rm sym} f(\xi) = G(\xi) f(\xi), \; G(\xi) = \int\limits_{\mathbb T^d} \hat a_{\rm sym} (\xi-\eta) \, \mu(\xi,\eta) \, d\eta
$$
apply component-wise.
Considering each component of $ r_{\boldsymbol{\ell}} = \{ r^i_{\boldsymbol{\ell}} \} $ separately and applying the Krein-Rutman theorem to the compact positivity improving operator $K_{\rm sym}$  (see \cite{KR}, Section 6,  or \cite{RS4}, Theorem XIII.44), we conclude that the operator
$\Big( G_{\rm sym}^{-1} K_{\rm sym} - \mathbb{E} \big)$ is invertible on $ L^2(\mathbb{T}^d) \ominus \{ \bf{1} \} $. Consequently,
equation \eqref{eps3} has a unique solution $r^i_{\boldsymbol{\ell}} \in  L^2(\mathbb{T}^d) \ominus \{ \bf{1} \}$.
Moreover,  \eqref{eps2} and \eqref{eps4} imply that
\begin{equation}\label{eps6}
\| r^i_{\boldsymbol{\ell}} \|_{L^2(\mathbb{T}^d)} \ \to \ 0 \quad \mbox{ as } \; {\boldsymbol{\ell}} \to 0,
\end{equation}
and thus \eqref{vareps} holds.

Next we substitute the right-hand side of \eqref{eps1} for $\tilde\varphi_0$ in \eqref{ER4biss}, and transform the resulting
relation with the help of  Proposition \ref{p_from_6} and  \eqref{eps6}. This yields
\begin{equation}\label{ER5}
\begin{array}{l}
\displaystyle
b^i ({\boldsymbol{\ell}})  =  \ell^j \, \int\limits_{\mathbb R^d}  \int\limits_{\mathbb T^d} z^i \, z^j \,  a_{\rm sym} (z) \,  \omega_{\boldsymbol{\ell}}(z) \, \mu (\xi, \xi-z )\, dz d\xi   \\ [7mm] \displaystyle  - \  2  \ell^j  \, \int\limits_{\mathbb R^d}  \int\limits_{\mathbb T^d} z^i \, \big( \varkappa^j_{\rm sym} (\xi-z) +   \frac12 \, r^j_{\boldsymbol{\ell}} (\xi-z) \big) \,
 a_{\rm sym} (z) \,  \omega_{\boldsymbol{\ell}}(z) \, \mu(\xi,\xi-z) \,  dz d\xi + O({|\boldsymbol{\ell}|}^2)  \\ [7mm] \displaystyle
 = \ell^j \, \int\limits_{\mathbb R^d}  \int\limits_{\mathbb T^d} z^i \, z^j \,  a_{\rm sym} (z)  \mu (\xi, \xi-z ) dz d\xi
 \\ [7mm] \displaystyle
 - \  2  \ell^j \, \int\limits_{\mathbb R^d}  \int\limits_{\mathbb T^d}  z^i \, \varkappa^j_{\rm sym} (\xi-z) \,
 a_{\rm sym} (z)  \mu(\xi,\xi-z)   dz d\xi + o({|\boldsymbol{\ell}|}).
\end{array}
\end{equation}
Since in the symmetric case $b(0) =0$ and $v_0 = \mathbf{1}$, then comparing \eqref{ER5} with \eqref{Theta}  and using one more time  the statement of Proposition \ref{p_from_6}  we come to \eqref{BER}.
\end{proof}

\noindent
{\bf \large Acknowledgement}

This work was completed during the visit of the second author at UiT, campus Narvik, in the Autumn 2018. The visit was supported by BFS/TFS project "Pure Mathematics in Norway".

\end{document}